\documentclass[
final
]{dmtcs-episciences}

\author{Michael Albert\affiliationmark{1}
  \and V\'{i}t Jel\'{i}nek\affiliationmark{2}
  \and Michal Opler\affiliationmark{2}}
\title[Two examples of Wilf-collapse]{Two examples of Wilf-collapse}
  
\affiliation{
  University of Otago, Dunedin, New Zealand\\
  Charles University, Prague, Czechia}
\received{2019-12-18}
\revised{2021-05-20}
\accepted{2021-05-21}
\publicationdetails{22}{2021}{2}{9}{5986}

\usepackage{hyperref}
\usepackage{amssymb}
\usepackage{tikz}
\usetikzlibrary{calc}
\usetikzlibrary{decorations}
\usetikzlibrary{decorations.pathreplacing}
\usetikzlibrary{shapes}
\usetikzlibrary{arrows, automata, decorations.pathreplacing, fit, matrix, patterns, positioning}
\usepackage{tikz-qtree}
\usepackage{xifthen} 
\usepackage{pgf}
\pgfmathsetseed{\number\pdfrandomseed}

\usepackage{amsmath, amsthm}

\usepackage{lineno}

\hypersetup{colorlinks=true, linkcolor=blue, urlcolor=blue, citecolor=blue, 
pdfpagemode=UseNone, pdfstartview=} 

\newcommand\absdot[2]{
	\node at #1 {\normalsize $\bullet$};
	\node at #1 [below] {$#2$};
}

\newcommand{\plotperm}[1]{
	\foreach \j [count=\i] in {#1} {
		\absdot{(\i,\j)}{};
	};
}

\newcommand{\plotpermgraph}[1]{
	\foreach \j [count=\i] in {#1} {
		\foreach \b [count=\a] in {#1} {
			\ifthenelse{\a<\i \AND \b>\j}{\draw (\a,\b)--(\i,\j);}{}
		};
	};
	\plotperm{#1};
}

\newcommand{\A}{{\mathcal A}}
\newcommand{\C}{{\mathcal C}}

\newcommand{\I}{{\mathcal I}}
\renewcommand{\L}{{\mathcal L}}

\newcommand{\M}{{\mathcal M}}

\renewcommand{\S}{{\mathcal S}}

\newcommand{\U}{{\mathcal U}}
\newcommand{\X}{{\mathcal X}}

\newcommand{\SIO}{{\mathcal{SIO}}} 

\renewcommand{\a}{\mathsf{a}} 
\renewcommand{\b}{\mathsf{b}} 
\newcommand{\w}{\mathsf{w}} 
\newcommand{\m}{\mathsf{m}} 

\newtheorem*{theorem*}{Theorem}
\newtheorem{theorem}{Theorem}[section]
\newtheorem{corollary}[theorem]{Corollary}
\newtheorem{lemma}[theorem]{Lemma}

\newtheorem{proposition}[theorem]{Proposition}

\theoremstyle{definition}

\newtheorem{observation}[theorem]{Observation}

\newcommand{\Av}{\operatorname{Av}}

\newcommand{\Inv}{\operatorname{Inv}}

\newcommand{\we}{\equiv}

\renewcommand{\leq}{\leqslant}

\newcommand{\bottom}{\bot} 

\begin{document}
\maketitle

\begin{abstract}
Two permutation classes, the X-class and subpermutations of the increasing oscillation are shown to exhibit an exponential Wilf-collapse. This means that the number of distinct enumerations of principal subclasses of each of these classes grows much more slowly than the class itself whereas \emph{a priori}, based only on symmetries of the class, there is no reason to expect this. The underlying cause of the collapse in both cases is the ability to apply some form of local symmetry which, combined with a greedy algorithm for detecting patterns in these classes, yields a Wilf-collapse.
\end{abstract}

\section{Introduction}

The coincidences that arise when two unrelated, or peripherally related collections of combinatorial structures have the same generating function have always attracted interest. In such cases one frequently seeks an explanation of this coincidence which might be an explicit size-preserving bijection between the classes, or parallel structural descriptions of the classes which frequently amount to an implicit description of such a bijection. In this paper we are concerned with certain families of such coincidences which are called \emph{Wilf-equivalence}. These arise when the two classes in question both occur as subsets of some universal class carrying a partial order (always interpreted as containment of structures, and presumed to satisfy the obvious conditions that such an order should have). Specifically if the universe is $\U$ and it carries a partial order $\leq$, then, to each $\alpha \in \U$ we associate the subclass of structures \emph{avoiding} $\alpha$, i.e.,
\[
\Av(\alpha) = \{ \tau \in \U \, : \, \alpha \not \leq \tau \}.
\]
We say that $\alpha$ and $\beta$ are \emph{Wilf-equivalent} (in $\U$) if $\Av(\alpha)$ and $\Av(\beta)$ have the same generating function.

More particularly, we are interested in situations where Wilf-equivalence is the rule, rather than the exception. For two structures to be Wilf-equivalent they must have the same size (since, for any $\alpha$ the lowest order term of the generating function for $\Av(\alpha)$ that differs from that of $\U$ is the term of degree $|\alpha|$ (where $|\alpha|$ denotes the size of $\alpha$). If, as $n \to \infty$ the number, $w_n$, the number of equivalence classes of the Wilf-equivalence relation on structures of size $n$ in $\U$ is small relative to the total number, $u_n$ of structures of size $n$ in $\U$ then we say that $\U$ exhibits a \emph{Wilf-collapse}. More specifically, a Wilf-collapse occurs if $w_n = o(u_n)$. If, for some constant $c < 1$, $w_n = o(c^n u_n)$ then we say that $\U$ exhibits an \emph{exponential Wilf-collapse}.

At the risk of some confusion we will refer to the equivalence classes of the Wilf-equivalence relation in a universe $\U$ as \emph{Wilf-equivalence classes}.

Here we will demonstrate an exponential Wilf-collapse in two permutation classes, $\X$ and $\SIO$ (defined below).  Our results complement those of another paper by the same authors, \cite{AJO:collapse-in-sum}, which includes some sufficient conditions to establish exponential Wilf-collapse. The independent interest of these two particular examples is that neither $\X$ nor $\SIO$ satisfies those conditions. However, both admit a natural representation of structures as words over an infinite alphabet which combines with a greedy algorithm to decide the containment relation to allow ``local'' applications of symmetries in creating Wilf-equivalent structures. It is in exploring this central thesis with new examples that we believe the value of this paper lies.

\subsection{Permutation classes}

In this subsection we give the briefest possible introduction to the basics of permutation classes as required for the remainder of the paper. For a much more comprehensive introduction we recommend \cite{Vatter-survey}.

Permutations of size $n$ are bijections from $[n] = \{1,2,\dots,n\}$ to itself, usually written in one line notation as $\pi = \pi_1 \pi_2 \dots \pi_n$. A permutation $\pi$ of size $k$ is \emph{contained in} another permutation $\tau$ if there is a subsequence of $k$ elements in $\tau$ whose relative ordering is the same as that of $\pi$. For instance 321 is contained in 3714526 because the subsequence 742 is in decreasing order. In this situation we write $\pi \leq \tau$ and also say that $\pi$ is \emph{involved in} $\tau$. If $\pi \not \leq \tau$ then we say that $\tau$ \emph{avoids} $\pi$.

The \emph{symmetries} of the universe $\S$ of all permutations are the order-automorphisms of $\S$. Concretely, there are eight of them generated by:
\begin{description}
  \item[reverse] which simply reverses the one-line representation of each permutation,
  \item[complement] which, in a permutation of size $n$ replaces each value $a$ in its one line notation by $n+1-a$, and
  \item[inverse] the normal inverse operation on bijections.
\end{description}

These are most naturally understood as operating on the set of points $(i, \pi_i)$ representing a permutation by reflection through a vertical, horizontal, or diagonal axis (and compositions of such reflections of course).

The \emph{sum}, $\sigma = \alpha \oplus \beta$, of two permutations $\alpha$ and $\beta$ is the permutation of size $|\alpha| + |\beta|$ where:
\[
\sigma(i) =
\left\{
\begin{array}{ll}
\alpha(i) & \text{for } 1 \leq i \leq |\alpha| \\
|\alpha| + \beta(i - |\alpha|)	& \text{for } |\alpha| < i  \leq |\alpha| + |\beta|
\end{array}
\right.
\]
That is, ``place $\alpha$ to the left of and below $\beta$''. If we change ``below'' to ``above'', then we will have defined the \emph{skew-sum}, $\alpha \ominus \beta$. A permutation is sum- (resp.~skew-) indecomposable if it cannot be written as a sum (skew-sum) of two permutations both of smaller size.

A \emph{permutation class} (or just \emph{class}), $\C$ is a collection of permutations downward-closed under containment, i.e., if $\tau \in \C$ and $\pi \leq \tau$, then $\pi \in \C$.  When we speak about Wilf-collapse it will always be within a universe that is some permutation class (specifically, the classes $\X$ and $\SIO$ introduced below). 

A class is \emph{sum-closed} if $\alpha, \beta \in \C$ implies that $\alpha \oplus \beta \in \C$. As above we use $\Av(\pi)$ to denote the set of permutations avoiding $\pi$ (possibly in a given class when the context is clear), and also $\Inv(\pi)$ to denote the set of permutations involving $\pi$.

We repeat now the definitions of Wilf-equivalence and Wilf-collapse in the context of permutation classes in particular. Throughout, we consider a particular class $\C$ to be given as our ``universe'' and all permutations mentioned are assumed to belong to $\C$ and all concepts are relative to $\C$. For instance, $\Av(\alpha) = \C \cap \{ \beta \, : \, \alpha \not \leq \beta \}$.

Two permutations $\alpha, \beta \in \C$ are \emph{Wilf-equivalent} if the generating functions of $\Av(\alpha)$ and $\Av(\beta)$ are the same. Equivalent conditions include: \begin{itemize}
 \item
 the generating functions of $\Inv(\alpha)$ and $\Inv(\beta)$ are the same,
 \item	
 there is a size preserving bijection between $\Av(\alpha)$ and $\Av(\beta)$,
 \item
 there is a size preserving bijection between $\Inv(\alpha)$ and $\Inv(\beta)$, and 
 \item
 there is a size preserving bijection from $\C$ to itself that sends $\Av(\alpha)$ to $\Av(\beta)$.
 \end{itemize}
In different contexts we have observed any one of these equivalent conditions to be the most natural one to prove. In $\X$ it will turn out to be easiest to compute (and compare) generating functions for $\Inv(\alpha)$ and $\Inv(\beta)$, while in $\SIO$ we demonstrate the existence of implicitly defined bijections.

Let $c_n$ denote the number of permutations in $\C$ of size $n$ and $w_n$ the number of Wilf-equivalence classes among those. Then we say that $\C$ exhibits a \emph{Wilf-collapse} if $w_n = o(c_n)$, and an \emph{exponential Wilf-collapse} if $w_n = o(r^n c_n)$ for some $0 < r < 1$.

In \cite{AJO:collapse-in-sum} we demonstrated some sufficient conditions for Wilf-collapse in sum-closed classes. One of these, Theorem 4.1, was:
\begin{theorem*}
\label{thm-sumclosed}
Any supercritical sum-closed class, $\C$,  that contains an incompatible pair has an exponential Wilf collapse, unless $\C$ is the class of increasing permutations.
\end{theorem*}

A few words in the statement of that theorem require further explanation: $\C$ is supercritical if the radius of convergence of the generating function of its sum-indecomposables is strictly greater than the radius of convergence of its generating function, and an incompatible pair is a pair of sum-indecomposable permutations $a$ and $b$ of $\C$ such that neither $a \oplus b$ nor $b \oplus a$ is a subpermutation of any sum-indecomposable element of $\C$.

The examples we consider in this paper have an exponential Wilf collapse. However,  the first one, $\X$, is not sum-closed and the second, $\SIO$ is sum-closed and super-critical but contains no incompatible pair. 

\section{The X-class}

The class $\X = \Av(2143, 2413, 3142, 3412)$ can be defined recursively as 1, together with all those permutations that start or finish with their maximum or minimum and that, when such an element is deleted, belong to $\X$. Alternatively, we can view it as the closure of $\{1\}$ under the four operations:
\[
\pi \mapsto 1 \oplus \pi, \: 
\pi \mapsto 1 \ominus \pi, \: 
\pi \mapsto \pi \oplus 1, \: 
\pi \mapsto \pi \ominus 1,
\]
or as the set of permutations that can be drawn on the diagonals of an axis-aligned square. The class $\X$ is closed under all the symmetries of $\S$. We will not explicitly state symmetrical variations of results below, but will feel free to use them as required.

Plainly, $\X$ is not sum- or skew-sum-closed and so the results of \cite{AJO:collapse-in-sum} cannot provide a Wilf-collapse. But, as we shall see there is a natural way of representing the elements of $\X$ as words over an (infinite) alphabet $\A$ and also a greedy method for determining when $\pi \leq \tau$ for $\pi, \tau \in \X$. Together, these provide a factorisation theorem for the generating functions of $\X \cap \Inv(\pi)$ in terms of the letters that occur in the word representing $\pi$. This is more than sufficient for an exponential Wilf-collapse.

The type of $\pi \in \X$ is either \emph{monotone} (if it is monotone increasing or decreasing) or \emph{crossed} (otherwise). If $\pi$ has size at least two and begins with its minimum or ends with its maximum we say it is \emph{positive} (or \emph{has positive sign}) while if it begins with its maximum or ends with its minimum it is \emph{negative}. Note that, according to these definitions, the permutation $1$ has no sign but every other permutation in $\X$ has a uniquely defined sign. For a non-negative integer $a$ we also use $a$ to denote the monotone increasing permutation of size $a$, and $-a$ to denote the monotone decreasing permutation of size $a$.

If $\pi \in \X$ of size $n$ is crossed and positive then, written in one-line notation, $\pi$ either begins with 1 or ends with $n$. In fact, for some $a, b \geq 0$ with $0 < a + b < n-1$, $\pi$ begins with $1 2 \dots a$ but not $1 2 \dots (a+1)$ and ends with $(n-b+1) (n-b+2) \dots n$ but not $(n-b) (n-b+1) \dots n$. That is, for this unique pair $(a,b)$ and some negative $\theta \in \X$, $\pi = a \oplus \theta \oplus b$. The conclusion of this discussion is:

\begin{observation}
\label{obs:shell}
If $\pi \in \X$ is crossed and positive then there exists a unique negative $\theta \in \X$ and a pair of non-negative integers $a, b$ with $a + b > 0$ such that $\pi = a \oplus \theta \oplus b$.
\end{observation}
%

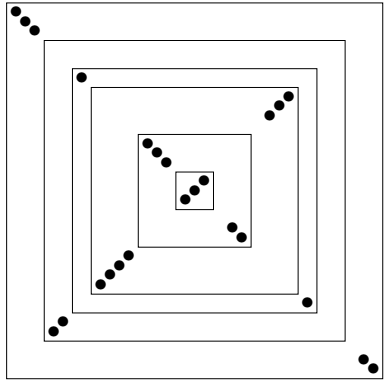
\begin{figure}
\centerline{
\begin{tikzpicture}[scale=0.25]
\draw[ultra thin] (0,0) rectangle (20,20);
\draw[ultra thin] (2,2) rectangle (18,18);
\draw[ultra thin] (3.5, 3.5) rectangle (16.5,16.5);
\draw[ultra thin] (4.5, 4.5) rectangle (15.5,15.5);
\draw[ultra thin] (7.0,7.0) rectangle (13.0,13.0);
\draw[ultra thin] (9.0,9.0) rectangle (11.0,11.0);
\absdot{(0.5,19.5)}{};
\absdot{(1,19)}{};
\absdot{(1.5,18.5)}{};
\absdot{(19,1)}{};
\absdot{(19.5,0.5)}{};
\absdot{(2.5,2.5)}{};
\absdot{(3.0,3.0)}{};
\absdot{(4.0,16.0)}{};
\absdot{(16.0,4.0)}{};
\absdot {(5.0,5.0)}{};
\absdot {(5.5,5.5)}{};
\absdot {(6.0,6.0)}{};
\absdot {(6.5,6.5)}{};
\absdot {(15.0, 15.0)}{};
\absdot {(14.5,14.5)}{};
\absdot {(14.0,14.0)}{};
\absdot{(7.5,12.5)}{};
\absdot{(8.0,12.0)}{};
\absdot{(8.5,11.5)}{};
\absdot{(12.0,8.0)}{};
\absdot{(12.5,7.5)}{};
\absdot{(9.5,9.5)}{};
\absdot{(10.0,10.0)}{};
\absdot{(10.5,10.5)}{};
\end{tikzpicture}
}
\label{fig:x-example}
\caption{A crossed and negative element of $\X$ whose word representation is $(-3,-2) (2,0) (-1,-1) (4,3) (-3,-2) (3)$}
\end{figure}

Let $\A^+$ be the set of all ordered pairs of non-negative integers except $(0,0)$ and let $\A^-$ be the set of all ordered pairs of non-positive integers except $(0,0)$. In either case, define the size of a pair, $(a,b)$, to be $|a| + |b|$. Finally let $\M$ be the set of all integers of absolute value at least 2 and define the size of an element of  $\M$ to be its absolute value. The \emph{sign} of an element of $\A^+$, $\A^-$, or $\M$ is defined in the obvious way. Define the language $\L$ to consist of:
\begin{itemize}
  \item
  The symbol $1$, of size $1$, and
  \item
  all words $w_1 w_2 \dots w_n$ for $n \geq 1$ satisfying the following conditions:
  \begin{itemize}
    \item
    $w_n \in \M$,
    \item
    $w_1, w_2, \dots, w_{n-1} \in \A^+ \cup \A^-$, and
    \item
    the signs of $w_1$, $w_2$, \dots, $w_n$ alternate. 
  \end{itemize}
\end{itemize}
The size of a word covered by the second case is defined to be the sum of the sizes of its letters.

Define $\Phi: \L \to \X$ recursively as follows:
\begin{align*}
\Phi(1) &= 1 \\
\Phi(a) &= a & \text{if $a \in \M$} \\
\Phi((a,b) v) &= a \oplus \Phi(v) \oplus b & \text{if $(a,b) \in \A^+$} \\
\Phi((a,b) v) &= a \ominus \Phi(v) \ominus b & \text{if $(a,b) \in \A^-$} \\
\end{align*}

See Figure \ref{fig:x-example} for an example of the action of $\Phi$ which will also be helpful in understanding the observation below.

\begin{observation}
The function $\Phi$ is a size-preserving bijection.
\end{observation}

\begin{proof}
Clearly, $\Phi$ is size-preserving. So, let $\omega \in \X$ be given. We provide an inductive proof that there is an element $w \in \L$ with $\Phi(w) = \omega$, and establish at the same time that $w$ is uniquely determined by $\omega$ which shows that $\Phi$ is a bijection.

The base cases are where $\omega$ is monotone. Since the signs of words in $\L$ having two or more letters alternate, no such word maps to a monotone permutation. On the other hand $\Phi(a) = a$ for $|a| \geq 2$ and $\Phi(1) = 1$ so each monotone permutation is the image of a unique element of $\L$.

If $\omega$ is positive and crossed (the negative case follows by symmetry of course) then, according to Observation \ref{obs:shell}, there is a unique pair, $(a,b)$,  of non-negative integers such that $\omega = a \oplus \nu \oplus b$ where $\nu$ is negative. By induction there is a unique $v \in \L$ whose first symbol is negative such that $\nu = \Phi(v)$. Then $\theta = \Phi((a,b) v)$. Conversely if $\Phi(w') = \omega$ and $w' = w_1' w_2' \dots w_m'$, then we must have $m \geq 2$ since $\omega$ is crossed. If $w_1' = (a', b')$ then, since: $\Phi(w') = a' \oplus \Phi(w_2' \dots w_m') \oplus b'$, $\Phi(w_2' \dots w_m')$ is negative, and Observation \ref{obs:shell}, it follows that $(a',b') = (a,b)$ and $\Phi(w_2' \dots w_m') = \nu$. By induction, $w' = w$ so $\Phi$ is one-to-one.
\end{proof}

Since $\Phi$ is a size-preserving bijection we freely identify $\L$ and $\X$. However, we would like to have an internal characterisation in $\L$ of the ordering that corresponds to the subpermutation ordering in $\X$ under this identification.

This is simple enough when we are looking for monotone subsequences of elements of $\X$. Given $a \geq 2$ and $t \in \L$, to check whether $a \leq \Phi(t)$ all we need to do is add up all the positive numbers occurring in letters of $t$ (in either coordinate, and in the element of $\M$ if it's positive) and add one more if the letter from $\M$ occurring in $t$ is negative. If the answer is $\geq a$ then $a \leq \Phi(t)$, and not otherwise. So, it remains to consider the crossed case.

In $\A^+$ and $\A^-$ we will take sums coordinatewise, i.e., $(a,b) + (c,d) = (a+c, b+d)$ in either case (but sums of letters of different types are not defined). The ordering on $\A^+$ and $\A^-$ is also to be taken coordinatewise but with respect to absolute value, i.e., $(a,b) \leq (c,d)$ if and only if $|a| \leq |c|$ and $|b| \leq |d|$. Let $w = w_1 w_2 \cdots w_k$ be any word over $\A^+ \cup \A^-$ (only the case of alternating signs will be of interest to us but for the following definition it's not required). The \emph{positive content} of $w$ is the sum of all its positive letters, and the \emph{negative content} of $w$ is the sum of all its negative letters.

Given $(a,b) \in \A^+$ and $w = w_1 w_2 \dots w_k w_{k+1} \in \L$, define the \emph{$(a,b)$-prefix} of $w$ to be the minimal prefix (if any) of $w$ whose positive content is greater than or equal to $(a,b)$ and define the \emph{$(a,b)$-remainder} of $w$ to be the remaining suffix after the $(a,b)$-prefix is removed -- both are undefined (denoted $\bottom$)  if $w$ has no prefix whose positive content is at least $(a,b)$.

The next proposition captures the essence of the argument that we can recognise involvement between crossed permutations and other elements of $\X$ greedily -- the uncrossed case has already been remarked upon above.

\begin{proposition}
Let negative $\tau, \theta \in \X$, $(a,b) \in \A^+$ and $(c,d)$ an arbitrary pair of non-negative integers  be given. Then
\[
a \oplus \theta \oplus b \leq c \oplus \tau \oplus d
\]
if and only if
\[
\max(0, a-c) \oplus \theta \oplus \max(0, b-d) \leq \tau.
\]
Moreover, if $\tau$ is monotone, then $a \oplus \theta \oplus b \not \leq \tau$, while if $\tau = c \ominus \nu \ominus d$ where $\nu$ is positive, then $a \oplus \theta \oplus b \leq \tau$ if and only if $a \oplus \theta \oplus b \leq \nu$.
\end{proposition}

\begin{proof}
The first part follows immediately from the observation that $1 \oplus \alpha \leq 1 \oplus \beta$ if and only of $\alpha \leq b$, and a symmetric variant for $\alpha \oplus 1$. Similarly, the non-trivial half of the second part follows from the observation that if $\alpha$ begins with its minimum or ends with its maximum, then $\alpha \leq 1 \ominus \beta$ if and only if $\alpha \leq \beta$ (and again, a symmetric variant for $\beta \ominus 1$.   
\end{proof}

\begin{corollary}
\label{cor:pref}
Let $w = (a,b) v \in \L$ and let $t \in \L$. Then $\Phi(w) \leq \Phi(t)$ if and only if the $(a,b)$-prefix of $t$ is defined and $\Phi(v) \leq \Phi(s)$ where $s$ is the $(a,b)$-suffix of $t$.
\end{corollary}

For $(a,b) \in \A^+ \cup \A^-$ set $P(a,b)$ to be the set of all words $w_1 w_2 \dots w_k$ over the alphabet $\A^+ \cup \A^-$ that satisfy:
\begin{itemize}
  \item $w_1$ has the same sign as $(a,b)$,
  \item the $w_i$ alternate in sign,
  \item the (sign of $w_1$) content of $w_1 w_2 \dots w_k$ is at least $(a,b)$, and
  \item the (sign of $w_1$) content of $w_1 w_2 \dots w_{k-1}$ is \emph{not} at least $(a,b)$.
\end{itemize}
Except for the sign restriction on $w_1$, the elements of $P(a,b)$ are exactly the candidates for being an $(a,b)$-prefix. Note that it is a consequence of the last two conditions that the sign of $w_k$ agrees with that of $w_1$ (i.e, $k$ is odd). Let $F_{(a,b)}(x)$ be the generating function for the set $P(a,b)$ enumerated by weight.

For a letter $a \in \M$, let $M_a(x)$ be the generating function for the set of all words whose corresponding permutations contain $a$.

Applying Corollary \ref{cor:pref} inductively we see that if $w = w_1 w_2 \dots w_n \in \L$, with $w_1$ positive and $t \in \L$ then $\Phi(w) \leq \Phi(t)$ if and only if $t$ can be factored as the concatenation of: an optional negative letter, followed by elements of $P(w_1)$, $P(w_2)$, through $P(w_{n-1})$ followed by a word that represents a permutation having a monotone subsequence $w_n$. Therefore, the generating function for the permutations containing $\Phi(w)$ is:
\[
\frac{1}{(1-x)^2} \left( \prod_{i=1}^{n-1} F_{w_i}(x) \right) M_{w_n}(x).
\]

By obvious symmetries $F_{(a,b)} = F_{(b,a)} = F_{(-a,-b)} = F_{(-b,-a)}$, and $M_a = M_{-a}$ so we obtain:

\begin{theorem}
  \label{thm:xcollapse}
  If $w = w_1 w_2 \dots w_n$ and $v = v_1 v_2 \dots v_n$ belong to $\L$, and there is a permutation $\rho : [n-1] \to [n-1]$ such that $w_i$ and $v_{\rho(i)}$ differ only in their sign and/or the order of their coordinates, and if $w_n = \pm v_n$ then the permutations $\Phi(w)$ and $\Phi(v)$ are Wilf-equivalent in $\X$.
\end{theorem}

The generating function for $\X$ (including the empty permutation) is 
\[
\frac{1 - 2x}{1-4x+2x^2}
\]
and the growth rate of the class is therefore $2 + \sqrt{2}$.

On the other hand, Theorem \ref{thm:xcollapse} implies that the number of Wilf-equivalence classes in $\X$ for structures of size $N \geqslant 2$ is dominated by the number of structures of size $N$ in another universe. In this universe, structures are made up of a multiset of atoms $(a,b)$ with $0 \leqslant a \leqslant b$ of weight $a+b$ and a single atom $m \geqslant 2$ of weight $m$. The size of a structure is the sum of the weights of its elements. Effectively then those structures correspond to coloured partitions where there are $\lceil (i+1)/2 \rceil$ different colours for parts of size $i$ corresponding to the pairs $(a,b)$ with $0 \leq a \leq b$ and $a + b = i$, together with a single part of size at least two. The generating function for this universe is:
\[
\frac{x^2}{1-x} \prod_{i=1}^{\infty} \left( \frac{1}{1-x^i} \right)^{\lceil (i+1)/2 \rceil}.
\]
After taking logarithms it is easy to verify that this converges for all $0 \leq x < 1$. By Pringsheim's theorem, the radius of convergence of the generating function for Wilf-equivalence classes in $\X$ is 1 (it certainly can't be greater than 1). In particular, the number of such Wilf-equivalence classes is $o(c^n)$ for all $c > 1$ 

Summarising:

\begin{theorem}
\label{thm:xwegf}
  Let $x_n = |\X_n|$ and let $w_n$ be the number of Wilf-equivalence classes in $\X_n$. Then,
  \[
  x_n^{1/n} \to 2 + \sqrt{2} \quad \text{and} \quad w_n^{1/n} \to 1 \quad \text{as} \quad n \to \infty.
  \]
\end{theorem}

The rules for Wilf-equivalence in $\X$ given by Theorem \ref{thm:xcollapse} are by no means complete. For instance, it is easy to check by considering avoidance that if $\theta$ and $\tau$ are Wilf-equivalent in $\X$ and neither ends with its maximum, then $1 \oplus \theta$ and $1 \oplus \tau$ are also Wilf-equivalent (see Figure \ref{fig:av1ptheta} for a ``proof without words''). So, for instance $1243$ and $1342$ are Wilf-equivalent in $\X$.

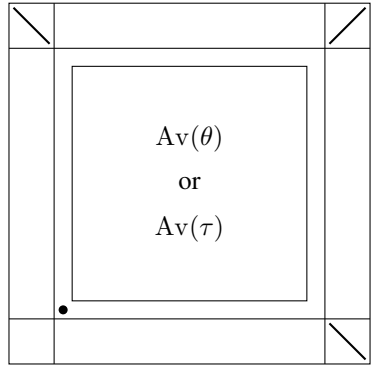
\begin{figure}
  \centerline{
  \begin{tikzpicture}[scale=0.6]
      \draw (-4,-4) rectangle (4,4);
      \draw[thick] (-3.9, 3.9) -- (-3.1,3.1);
      \draw[thick] (3.9, 3.9) -- (3.1,3.1);
      \draw[thick] (3.9, -3.9) -- (3.1,-3.1);
      \draw (-3, -4) -- (-3,4);
      \draw (3, -4) -- (3,4);
      \draw (-4, -3) -- (4,-3);
      \draw (-4, 3) -- (4,3);
      \fill (-2.8,-2.8) circle (0.1);
      \node (A) at (0,1) {$\Av(\theta)$};
      \node (B) at (0,0) {or};
      \node (C) at (0,-1) {$\Av(\tau)$};
      \draw (-2.6, -2.6) rectangle (2.6, 2.6);      
  \end{tikzpicture}  
  }
  \caption{A proof by pictures that, if if $\theta$ and $\tau$ are Wilf-equivalent in $\X$ and neither ends with its maximum, then $1 \oplus \theta$ and $1 \oplus \tau$ are also Wilf-equivalent. The large box represents an element of either $\Av(1 \oplus \theta)$ or $\Av(1 \oplus \tau)$. If the element indicated by a $\bullet$ is not present, then the permutation avoids both. If it is present, then the central region must avoid $\theta$ (respectively, $\tau$) and the bijection witnessing their Wilf-equivalence can be applied to it.}
  \label{fig:av1ptheta}
\end{figure}

\section{Sum closure of increasing oscillations}


We now consider an important class $\SIO$ ($\S$ubpermutations of the $\I$ncreasing $\mathcal{O}$scillation) to which the previous conditions do not apply and yet we can still demonstrate an exponential Wilf-collapse. This class consists of all finite permutations whose pattern can be found in the \emph{infinite increasing oscillation}:
\[
 2 ,  \, 4 , \, 1 , \, 6 , \, 3 , \, 8 , \, 5 , \, 10 , \, 7 , \, \dots
\]
This sequence (and various minor modifications of it) arises in a wide variety of applications ranging from the recreational e.g., the theory of juggling \cite{Buhler} to the more serious in that it provides those sequences that are most difficult to sort by block reversals \cite{BafnaPevzner}. The class $\SIO$ also plays a central role in the study of growth rates of permutation classes providing a fundamental building block for the construction of intervals in the set of achievable growth rates and other threshold phenomena \cite{Bevan, PantoneVatter, Vatter-above}. The class $\SIO$ is invariant under four of the symmetries for permutations in general: the identity, reverse-complement, inverse, and the composition of reverse-complement with inverse.

The \emph{inversion graph} of a sequence or permutation $\pi$ is the graph on the indices of the sequence where, for $i < j$, $i$ and $j$ are adjacent if $\pi_i > \pi_j$, i.e., if the pair $(i,j)$ forms an inversion in the sequence.
The inversion graph of the infinite increasing oscillation is a (one-ended) infinite path and the finite permutations in the class $\SIO$ have sum decompositions whose components have inversion graphs that are paths (possibly consisting only of a single element). For each path of length at least three there are two permutations whose inversion graph is that path and both of these occur in $\SIO$. So the alphabet $\A$ of sum-indecomposables for the class $\SIO$ consists of: a symbol $\a$ representing the singleton permutation ($\a$), a symbol $\b$ representing 21, and for $k \geq 3$ two symbols $\w_k$ and $\m_k$ representing the two different permutations of size $k$ whose inversion graph is a path. The first few of these are shown explicitly in Figure \ref{fig-sum-indec-in-inc-osc} along with a convenient symbolic representation of each. 

In the sum-indecomposables of size at least 3 we define the \emph{Start} of that sum indecomposable to be ``Up'' ($\nearrow$) if the initial element of the permutation corresponds to an interior vertex of the inversion graph, and ``Down'' ($\searrow$) if it is a terminal vertex. The corresponding definition for \emph{Finish} is that it is $\nearrow$ if the final element of the permutation corresponds to an interior vertex of the inversion graph, and $\searrow$ if it is a terminal vertex. Start and Finish are ``Undefined'' ($\bottom$) for the permutations $1$ or $21$. 

Since every permutation in $\SIO$ is uniquely represented as a sum of sum-indecomposables and all such sums belong to $\SIO$ we freely identify $\SIO$ with words over the alphabet $\A$. We further define the Start of a word to be the Start of its first letter (possibly undefined) and its Finish to be the Finish of its terminal letter. The \emph{type} of a word is the pair consisting of its Start and Finish.

\begin{figure}
\begin{tabular}{cccccc}
Sum indecomposable & Diagram & Symbol & Letter & Start & Finish \\ \hline
1 &
\begin{tikzpicture}[scale=0.25]
\plotperm{1}
\end{tikzpicture}
&
\begin{tikzpicture}[scale=0.5]
\absdot{(0,0)}{}
\end{tikzpicture}
&
$\a$ 
&
$\bottom$
&
$\bottom$
\\
21 &
\begin{tikzpicture}[scale=0.25]
\plotperm{2,1}
\draw[dashed] (1,2) -- (2,1);
\end{tikzpicture}
&
\begin{tikzpicture}[scale=0.5]
\draw[thick] (0,1) -- (0,0) ;
\end{tikzpicture}
&
$\b$ 
&
$\bottom$
&
$\bottom$
\\
231 &
\begin{tikzpicture}[scale=0.25]
\plotperm{2,3,1}
\draw[dashed] (1,2) -- (3,1) -- (2,3);
\end{tikzpicture}
&
\begin{tikzpicture}[scale=0.5]
\draw[thick] (0,1) -- (1,0) -- (2,1);
\end{tikzpicture}
&
$\w_3$ 
&
$\searrow$
&
$\nearrow$
\\
312 &
\begin{tikzpicture}[scale=0.25]
\plotperm{3,1,2}
\draw[dashed] (2,1) -- (1,3) -- (3,2);
\end{tikzpicture}
&
\begin{tikzpicture}[scale=0.5]
\draw[thick] (0,0) -- (1,1) -- (2,0);
\end{tikzpicture}
&
$\m_3$
&
$\nearrow$
&
$\searrow$
\\
2413 &
\begin{tikzpicture}[scale=0.25]
\plotperm{2,4,1,3}
\draw[dashed] (1,2) -- (3,1) -- (2,4) -- (4,3);
\end{tikzpicture}
&
\begin{tikzpicture}[scale=0.5]
\draw[thick] (0,1) -- (1,0) -- (2,1) -- (3,0);
\end{tikzpicture}
&
$\w_4$
&
$\searrow$
&
$\searrow$
\\
3142 &
\begin{tikzpicture}[scale=0.25]
\plotperm{3,1,4,2}
\draw[dashed] (2,1) -- (1,3) -- (4,2) -- (3,4);
\end{tikzpicture}
&
\begin{tikzpicture}[scale=0.5]
\draw[thick] (0,0) -- (1,1) -- (2,0) -- (3,1);
\end{tikzpicture}
&
$\m_4$
&
$\nearrow$
&
$\nearrow$
\\
24153 &
\begin{tikzpicture}[scale=0.25]
\plotperm{2,4,1,5,3}
\draw[dashed] (1,2) -- (3,1) -- (2,4) -- (5,3) -- (4,5);
\end{tikzpicture}
&
\begin{tikzpicture}[scale=0.5]
\draw[thick] (0,1) -- (1,0) -- (2,1) -- (3,0) -- (4,1);
\end{tikzpicture}
&
$\w_5$
&
$\searrow$
&
$\nearrow$
\\
31524 &
\begin{tikzpicture}[scale=0.25]
\plotperm{3,1,5,2,4}
\draw[dashed] (2,1) -- (1,3) -- (4,2) -- (3,5) -- (5,4);
\end{tikzpicture}
&
\begin{tikzpicture}[scale=0.5]
\draw[thick] (0,0) -- (1,1) -- (2,0) -- (3,1) -- (4,0);
\end{tikzpicture}
&
$\m_5$
&
$\nearrow$
&
$\searrow$
\end{tabular}
\caption{
The sum indecomposables of size at most 5 in the class $\SIO$ together with their diagrams, symbolic representation, representative letters of the alphabet $\A$ and start and finish type (defined in text).
}
\label{fig-sum-indec-in-inc-osc}
\end{figure}

The symbolic representation of permutations in $\SIO$ makes it particularly easy to identify when one permutation is involved in another -- more specifically when a word is involved in a single letter. The basic idea is that you must try to pack the individual parts of the word into the letter (and may do so greedily from left to right) using the following conventions:
\begin{itemize}
\item
The letter $\a$ can be matched to any single vertex,
\item
The letter $b$ can be matched to any edge,
\item
Any letter $\w_k$ or $\m_k$ can be matched to a connected sequence of $k-1$ edges that starts with an edge of the appropriate  slope.
\end{itemize}
All this is subject to the condition that the parts of the symbol matching consecutive letters cannot be connected by an edge. For instance the relationship $\w_3 \a \m_4 \b \leq \m_{16}$ (and implicitly $\w_3 \a \m_4 \b \not\leq \m_{14}$) is shown by the thicker edges (and dot) below:

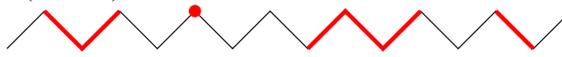
\centerline{
\begin{tikzpicture}[scale = 0.5]
\draw (0,0) -- (1,1) -- (2,0) -- (3,1) -- (4,0) -- (5,1) -- (6,0) -- (7,1) -- (8,0) -- (9,1) -- (10,0) -- (11,1) -- (12,0) -- (13,1) -- (14,0) -- (15,1);
\draw[ultra thick, red] (1,1)--(2,0)--(3,1);
\draw[red, fill] (5,1) circle (0.15);
\draw[ultra thick, red] (8,0)--(9,1)--(10,0)--(11,1);
\draw[ultra thick, red] (13,1)--(14,0);
\end{tikzpicture}}

The symbolic viewpoint also makes it easy to understand the effect on $\SIO$ of the symmetries: reverse-complement, inverse, and their composition. Respectively, they correspond to: rotation by $180^{\circ}$, reflection in a horizontal axis, and reflection in a vertical axis.

\begin{observation}
The class $\SIO$ is supercritical since the generating function for its sum-closed permutations is:
\[
x + x^2 + 2 x^3/(1-x).
\]
However, it has no incompatible pairs since for any pair of plus-indecomposable permutations $\pi, \theta \in \SIO$, $\pi \oplus \theta$ occurs as a pattern in a sum-indecomposable of size at most $|\pi| + |\theta| + 2$.
\end{observation}

Define an equivalence relation $\sim$ on $\A^*$ by $X \sim Y$ if and only if there is a size and type-preserving bijection $\Phi : \SIO \to \SIO$ such that the image of $\SIO \cap \Av(X)$ is $\SIO \cap \Av(Y)$ (where again as usual we make no distinction between the words $X$ and $Y$ and the permutations they represent). In particular if $X \sim Y$ then $X \we_{\SIO} Y$. So, if we can demonstrate a significant collapse for the $\sim$ equivalence relation on $\SIO$ then we will have demonstrated a Wilf-collapse in $\SIO$ as well.

\begin{lemma}
\label{symmetry}
Let $X \in \A^*$ and a symmetry $\sigma$ be given such that both the Start and Finish of $X$ are defined and equal the Start and Finish of $\sigma(X)$. Then $X \sim \sigma(X)$
\end{lemma}

\begin{proof}
The result would be trivial if we only required a size-preserving bijection since $\sigma$ is such a bijection between  $\SIO \cap \Av(X)$ and $\SIO \cap \Av(\sigma(X))$. However, the required construction is not that much more complex. Without loss of generality suppose that the type of $X$ is $(\nearrow, \searrow)$ (all other cases can be proven in the same way). In this case $\sigma$ must be the symmetry whose action on symbols is reflection in a vertical axis. This symmetry has the effect of exchanging and inverting the Start and Finish of a word.

If $W \in \{ \a, \b, \w_3 \}^*$ just set $\Phi(W) = W$ (since the Start and Finish of $X$ are $\nearrow$ and $\searrow$ respectively, all such $W$ avoid both $X$ and $\sigma(X)$). Otherwise let $W = PES$ where $P$ is the maximal prefix of $W$ in $\{ \a, \b, \w_3 \}^*$ and $S$ is the maximal suffix in $\{ \a, \b, \w_3 \}^*$ (either might be empty). The word $E$ (the ``essential'' part of $W$) has a defined Start and Finish. If $\Phi(E)$ is defined (having the same type and size as $E$ and behaving correctly in the sense that $X \leq E$ if and only if $\sigma(X) \leq \Phi(E)$) then we can set $\Phi(W) = P \Phi(E) S$ since this preserves type and size and $X \leq W$ if and only if $X \leq E$.

So let a word $E$ in $\A^*$ having defined Start and Finish be given that neither begins nor ends with $\w_3$. Define the word $E'$ whose diagram is obtained from the diagram of $E$ by deleting the first edge if the Start of $E$ is not $\nearrow$ and the last edge if the Finish of $E$ is not $\searrow$. Because $E$ does not begin or end with $\w_3$ the Start of $E'$ is $\nearrow$ and its Finish is $\searrow$. Note also that $X \leq E$ if and only if $X \leq E'$. Finally, to define $\Phi(E)$ just take $\sigma(E')$ and reattach the edges that were deleted. For example, if $E = \w_7 \a \m_5$ then the following series of diagrams shows the construction of $\Phi(E) = \w_6 \a \w_6$ (the dots on the endpoint of the edge that is deleted/added are for emphasis only):
\begin{align*}
E &= \w_7 \a \m_5 &=\begin{tikzpicture}[scale = 0.5]
\draw[fill] (-1,1) circle (0.0);
\draw[fill] (0,1) circle (0.1);
\draw (0,1) -- (1,0) -- (2,1) -- (3,0) -- (4,1) -- (5,0) -- (6,1);
\draw[fill] (7, 0.5) circle (0.1);
\draw (8,0) -- (9,1) -- (10,0) -- (11,1) -- (12,0);
\end{tikzpicture} \\
E' &= \m_6 \a \m_5 &= \begin{tikzpicture}[scale = 0.5]
\draw[fill] (-1,1) circle (0.0);
\draw[fill] (0,1) circle (0.0);
\draw (1,0) -- (2,1) -- (3,0) -- (4,1) -- (5,0) -- (6,1);
\draw[fill] (7, 0.5) circle (0.1);
\draw (8,0) -- (9,1) -- (10,0) -- (11,1) -- (12,0);
\end{tikzpicture} \\
\sigma(E') &= \m_5 \a \w_6 &=
\begin{tikzpicture}[scale = 0.5]
\draw[fill] (-1,1) circle (0.0);
\draw (1,0) -- (2,1) -- (3,0) -- (4,1) -- (5,0);
\draw[fill] (6, 0.5) circle (0.1);
\draw (7,1) -- (8,0) -- (9,1) -- (10,0) -- (11,1) -- (12,0);
\end{tikzpicture}  \\
\Phi(E) &= \w_6 \a \w_6 &=
\begin{tikzpicture}[scale = 0.5]
\draw[fill] (-1,1) circle (0.0);
\draw[fill] (0,1) circle (0.1);
\draw (0,1) -- (1,0) -- (2,1) -- (3,0) -- (4,1) -- (5,0);
\draw[fill] (6, 0.5) circle (0.1);
\draw (7,1) -- (8,0) -- (9,1) -- (10,0) -- (11,1) -- (12,0);
\end{tikzpicture}
\end{align*}

The following table shows how to adjust the construction for each possible type of $X$ -- the main point is that we remove the maximal unusable prefix (with respect to including $X$) and the maximal unusable suffix (ditto) and then operate on the remaining essential part.

\centerline{
\begin{tabular}{ccc}
Type of $X$ & Prefix & Suffix \\ \hline
$(\nearrow, \nearrow)$ & $\{\a, \b, \w_3\}^\ast$ & $\{\a, \b, \m_3\}^\ast$ \\
$(\nearrow, \searrow)$ & $\{\a, \b, \w_3\}^\ast$ & $\{\a, \b, \w_3\}^\ast$ \\
$(\searrow, \nearrow)$ & $\{\a, \b, \m_3\}^\ast$ & $\{\a, \b, \m_3\}^\ast$ \\
$(\searrow, \searrow)$ & $\{\a, \b, \m_3\}^\ast$ & $\{\a, \b, \w_3\}^\ast$
\end{tabular}
}
\end{proof}

Now comes the main ingredient of the collapse -- a substitution principle:

\begin{proposition}
\label{prop-substitution}
Suppose that $X \sim Y$ and that $P, S \in \A^*$. Then $P X S \sim P Y S$.
\end{proposition}

\begin{proof}
Let $\Phi : \SIO \to \SIO$ be a size and type-preserving bijection whose existence is guaranteed by the condition that $X \sim Y$. We will construct a size and type-preserving bijection $\Psi: \SIO \to \SIO$ which witnesses $PXS \sim PYS$.

To this end, let $W \in \A^*$ be given. If $PS \not \leq W$ set $\Psi(W) = W$. So suppose that $PS \leq W$. Take a leftmost embedding of $P$ into $W$ and a rightmost embedding of $S$ into $W$. If $P$ is non-empty then the embedding of $P$ ends in some letter $x$ and there might remain some ``usable'' part of $x$ (beginning with the next vertex but one after the end of $P$) call it $x'$ (if $P$ is empty then $x$ will just be the first character of $W$ and $x' = x$). Likewise $S$ starts in some letter $y$ and there might be some usable part of $y$ before the beginning of $S$, call it $y'$. Suppose first that $x \neq y$ (not just as characters but as letters in the word $W$) and write $W = A x M y B$. Consider $W' = x' M y'$. Now $PXS \leq W$ if and only if $X \leq W'$. So $PXS \leq W$ if and only if $Y \leq \Phi(W')$. Since the Start and Finish of $\Phi(W')$ are the same as of $W'$ we can reattach the part of $x$ before $W'$ and the part of $y$ after $W'$ to $\Phi(W')$ in the same way as they were originally attached to $W$. Define $\Psi(W)$ to be the concatenation of $A$, this reattachment, and $B$. If $P$ is non-empty then $Ax$ is non-empty and the Start of $\Psi(W)$ is the same as the Start of $Ax$ and hence of $W$. On the other hand if $P$ is empty then the Start of $W$ is just the Start of $x'$ (and so of $W'$), and once again the Start of $\Psi(W)$ is the same as that of $W$. Similar considerations show that the Finish of $\Psi(W)$ is the same as the Finish of $W$ and complete this case.

However, the case when $x$ any $y$ are the same character is essentially the same, but $W'$ just consists of a single letter which is the usable part of the letter in which $P$ ends and $S$ begins.
\end{proof}

\begin{theorem}
The class $\SIO$ has an exponential Wilf-collapse.
\end{theorem}

\begin{proof}
As noted previously, the generating function for the plus indecomposable members of $\SIO$ is
\[
A(t) = t + t^2 + \frac{2 t^3}{1- t}
\]
which is rational and of radius of convergence $1$ and therefore the class $\SIO$ is super-critical.

Proposition 3.10 of \cite{AJO:collapse-in-sum}  implies that for some constant $c > 0$, in all but an exponentially small proportion of the words in $\A^*$ of weight $n$ there are at least $c n$ disjoint occurrences of the factor $\w_3 \m_4$. However $\w_3 \m_4 \sim \w_4 \m_3$ so the $\sim$ (and hence Wilf) equivalence classes of all such words contain at least $2^{cn}$ elements, which demonstrates an exponential Wilf-collapse.
\end{proof}

\section{Conclusion}

This paper provides more evidence for a thesis that has previously been explored in \cite{AlbertBouvel, AlbertLi, AJO:collapse-in-sum}. Put most broadly:
\begin{quote}
\textit{If the elements of a class $\C$ admit a natural representation as words over some weighted alphabet in such a way that containment can be recognised at the word level in a greedy fashion, then global symmetries of the class can often be applied locally to provide Wilf-equivalent structures and this frequently implies a Wilf collapse.	}
\end{quote}

However, we have no other general criteria that yield Wilf-collapse. 

The class $\Av(321)$ provides a testing ground for this thesis for a couple of reasons. Results in \cite{AlbertLackner} show that there is a greedy form of testing for containment within this class. On the other hand, results in \cite{ABRV_321} also give a natural (or at least arguably natural) encoding of the elements of the class in the form of words. This encoding implies a restricted form for the generating functions of the classes $\Av(321, \pi)$, namely that they must be rational. These are exactly the classes where we would hope to see many coincidences between the generating functions if a Wilf-collapse exists in $\Av(321)$. However, the two results do not fit together well enough to allow the sort of manipulations that we have seen here, so this case remains an intriguingly open one.

\subsection{Acknowledgement}

The authors are grateful to an anonymous referee for suggestions that have greatly improved the presentation and clarity of some of the material.

\bibliographystyle{plain}
\bibliography{sio}

\end{document}